\documentclass{article}
\usepackage[utf8]{inputenc}
\usepackage{amsmath, amsthm, amssymb, faktor, graphicx, mathabx, bbm}
\usepackage{thm-restate}
\usepackage{caption}
\usepackage{float,subfig}
\usepackage{enumerate}
\usepackage{hyperref}
\hypersetup{
    colorlinks=true,
    linkcolor=blue,
    filecolor=magenta,      
    urlcolor=cyan,
}

\usepackage[all]{xy}
\usepackage{color}
\usepackage{ytableau}

\usepackage[utf8]{inputenc}

\DeclareFixedFont{\ttb}{T1}{txtt}{bx}{n}{9} 
\DeclareFixedFont{\ttm}{T1}{txtt}{m}{n}{9}  

\usepackage{xcolor}
\definecolor{keyword}{rgb}{0,0,0.5}
\definecolor{emph}{rgb}{0.6,0,0}
\definecolor{string}{rgb}{0,0.5,0}

\usepackage{listings}

\usepackage{tikz}
\usetikzlibrary{calc,intersections,backgrounds,cd}

\theoremstyle{plain}

\newtheorem{theorem}{Theorem}[section]

\newtheorem{lemma}[theorem]{Lemma}

\newtheorem{definition}[theorem]{Definition}

\newtheorem{construction}[theorem]{Construction}

\theoremstyle{definition}

\newtheorem{example}[theorem]{Example}

\theoremstyle{remark}

\theoremstyle{definition}

\usepackage[
backend=biber,
style=alphabetic,
sorting=ynt
]{biblatex}

\def\scat {\mathcal{C}}
\def\gs {\mathfrak{C}}
\def\Obj {\mathrm{Obj}}
\def\Arr {\mathrm{Arr}}
\def\bh {\scat\Lambda}
\def\cg {\Lambda}

\def\idGrp {\mathbf{1}}

\def\GL {\mathrm{GL}}
\def\L {\widehat{L}}
\def\T {T}
\def\chat {\widehat{\Psi}_{G}}
\def\lchat {\widehat{\Psi}_{G_{\sigma}}}

\newcommand{\sheafpath}[1]{M_{#1}}

\newcommand{\idCat}[1]{\mathbbm{1}_{#1}}

\title{Sheaves over complexes of groups and developability}
\author{Joshua L. Faber \\ \href{mailto:joshlfaber@gmail.com}{joshlfaber@gmail.com}}

\date{\today}

\begin{document}

\maketitle
\begin{abstract}
We define the notion of a sheaf over a complex of groups.  As an application, we give a criterion for the developability of a complex of groups. When the developability is witnessed by a morphism to $\mathrm{GL}(V)$ for some $V$, our criterion is a characterization
of developability.
\end{abstract}

\section{Introduction}
Motivated by Joel Friedman's construction of sheaves over digraphs \cite{MR3289057} we construct sheaves over categories. Our definition is deceptively simple to state. 

 \begin{restatable}[]{definition}{Sheaf} \label{D: sheaf}
 Let $\scat$ be a category, $\mathbb{K}$ a field, and $\mathbb{K}$-Vect the category of $\mathbb{K}$-vector spaces. A \textbf{sheaf} over $\scat$ is a functor $\mathcal{F}:\scat\rightarrow \mathbb{K}$-Vect.
 \end{restatable}
 In \cite[III.$\scat$  2.1 p.535]{MR1744486} Bridson and Haefliger define a complex of groups $\cg$ and in \cite[III.$\scat$ 2.8 p.538]{MR1744486} define what we call the \textit{B.H. category} $\bh$. The B.H. category $\bh$ is constructed from $\cg$ and encodes the data of $\cg$ into the structure of the space $\bh$ so that $\pi_1(\cg) \cong \pi_1(\bh)$. By placing a sheaf $\mathcal{F}$ over $\bh$ we provide a functor from $\bh$ to the category $\mathbb{K}$-Vect. Functors from the B.H. category $\bh$ of a complex of groups $\cg$ are not a new area of research as the next two examples show. Let $\cg$ be a complex of groups, $\bh$ the B.H. category and $\mathrm{Ab}$ the category of abelian groups. In \cite[4.1.5 p.298]{haefliger1992extension} Haefliger defines a $\cg$-module $\mathcal{M}$ as a functor $\mathcal{M}:\cg\rightarrow \mathrm{Ab}$. Let $\cg$, $\cg'$ be complexes of groups, $\bh$, $\bh'$ their B.H categories and $\phi:\cg\rightarrow \cg'$ a morphism of complexes of groups. In \cite[III.$\scat$. 2.4 p.536]{MR1744486} Bridson and Haefliger show how $\phi$ induces a functor $\widehat{\phi}:\bh\rightarrow\bh'$ between B.H categories which carries the same data as $\phi$ but translated into category theory. 
 
 Unlike graphs of groups from Bass--Serre theory, complexes of groups cannot always be realized as the result of a global quotient. When $\cg$ \textit{can} be realized as the result of a global quotient we say that $\cg$ is \emph{developable}. In Bass--Serre Theory a graph of groups is developable when one can construct a Bass--Serre tree. The \emph{fundamental theorem of Bass--Serre theory} states that all graphs of groups are developable \cite[III.$\scat$ 2.17 p.544]{MR1744486}. When $\cg$ is developable the usual covering space theory of spaces works when we translate to the side  of the B.H. category $\bh$. In \cite[III.$\scat$ 2.15 p.544]{MR1744486} Bridson and Haefliger characterize the development of a complex of groups $\cg$ in terms of a group $G$ which is a \emph{witness} to the development of $\cg$, see Theorem \ref{key theorem}.
 
 In this paper we show that when the group $\GL(V)$ is a witness to the developability of $\cg$ we can characterize the developability of $\cg$ in terms of a sheaf $\mathcal{F}$ that satisfies the following properties. Let $\cg$ be a complex of groups, $\bh$ the B.H. category and $\mathcal{F}$ a sheaf. If there exists an $n\in \mathbb{N}$ so that for each $\sigma\in \Obj(\bh)$ we have $\dim(\mathcal{F}(\sigma)) = n$ then $\mathcal{F}$ has \textit{constant rank}. For $\sigma\in \Obj(\bh)$ let $S_{\sigma}\subseteq \bh$ be the full subcategory over $\sigma$. If for each $\sigma\in \Obj(\bh)$ the restriction $\mathcal{F}|_{S_{\sigma}}:S_{\sigma}\rightarrow \mathbb{K}\mathrm{-Vect}$ is injective then $\mathcal{F}$ is \textit{injective on local groups}. We say that $\mathcal{F}$ satisfies the \emph{dev properties} if \emph{i}) $\mathcal{F}$ has constant rank, \emph{ii}) for each morphism $\lambda\in\mathrm{Mor}(\bh)$ the linear map $\mathcal{F}(\lambda)$ is invertible, and \emph{iii}) $\mathcal{F}$ is injective on local groups. We can now state our main theorem. 
 
\begin{restatable}[]{thmA}{Dev} \label{main}
Let $\cg$ be a complex of groups and $\bh$ the B.H. category. The group $\mathrm{GL}(V)$ is a witness to the developability of $\cg$ if and only if there exists a sheaf $\mathcal{F}$ over $\bh$ satisfying the dev-properties.
\end{restatable}
\subsection{Acknowledgements}
We would like to thank Michelle Chu whose suggestions helped inspire Lemma \ref{onedirection} and to Daniel Groves whose continuous support helped guide this paper. 
\section{Background}
Scwols were first introduced by Bridson and Haefliger in \cite[III.$\scat$ 1 p.520]{MR1744486} to define a complex of groups construction. Complexes of groups are abstract objects which are natural generlizations of graphs of groups. Given a polyhedral complex $\mathcal{M}$ there is a process that constructs a scwol $\mathcal{X}$ from $\mathcal{M}$ and if given a group $G$ which acts on $\mathcal{M}$ we can construct a complex of groups $\cg$ from the quotient space $G\setminus \mathcal{X}$. There is no loss in data when constructing $\cg$ over a scwol $\gs$ in place of a polyhedral complex $\mathcal{M}$. We start with the definition of a scwol $\gs$ and the construction of a scwol from a complex $\mathcal{M}$. We then define complexes of groups $\cg$ as well as their associated B.H. categories $\bh$ and give results concerning these gadgets. The following can be found in \cite[III.$\scat$ 1 p 520]{MR1744486},
\begin{definition}
A \textbf{scwol} $\mathcal{X}$, is a small category such that,
\begin{itemize}
    \item for all composable arrows $a,b\in \Arr(\mathcal{X})$ if $i(a) \neq t(a)$ and $i(b) \neq t(b)$ then $i(ab) \neq t(ab)$; and
    \item for all arrows $a\in \Arr(\mathcal{X})$ if $i(a) = t(a) $ then $a$ is the identity morphism. 
\end{itemize}
\end{definition}
\begin{definition}
Let $\mathcal{X}$ be a scwol, $k \in \mathbb{N}$ a natural number then $E^{(k)}(\gs)$ is the set of sequences $(a_1, \cdots, a_n)$ of composable arrows in $\mathcal{X}$ where $i(a_l) = t(a_{l+1})$ for $1\leq l < k$.  
\end{definition}
\vspace{.1in}

Given a polyhedral complex $\mathcal{M}$ we construct a scwol through the following construction, 

\begin{construction}\label{scwol construction}
Let $\mathcal{M}$ be a polyhedral complex. The \textbf{complex to scwol construction} $\mathfrak{M}$ takes $\mathcal{M}$ and constructs the scwol $\mathfrak{M}(\mathcal{M})$ by taking the first barycentric subdivision of $\mathcal{M}$ and,
\begin{itemize}
    \item associating the barycenter of each cell $\sigma \in \mathcal{M}$ to an object $\sigma \in \Obj(\mathfrak{M}(\mathcal{M}))$; and 
    \item associating each $1$-simplex of the barycentric subdivision of $\mathcal{M}$ to an arrow $a\in \Arr(\mathfrak{M}(\mathcal{M}))$. 
\end{itemize}
\end{construction}
 A directed graph $\Gamma$ is a $1$-dimensional polyhedral complex. In \cite[1 p.6]{MR3289057} Joel Friedman defines a sheaf over $\Gamma$ by placing a $\mathbb{F}$-vector space over each vertex, a $\mathbb{F}$-vector space over each edge, and a linear morphism over each attaching map. We give Friedman's definition below, 
 
 \begin{definition}[Joel Friedman, {\cite[1 p.6]{MR3289057}}]
A \textbf{sheaf} $\mathcal{F}$ of finite dimensional $\mathbb{F}$-vector spaces over a digraph $\Gamma$ is,
\begin{itemize}
    \item[(1)]  a finite dimensional $\mathbb{F}$-vector space $\mathcal{F}(v)$ for each $v\in V\Gamma$;
    \item[(2)] a finite dimensional $\mathbb{F}$-vector space $\mathcal{F}(e)$ for each $e\in E\Gamma$;
    \item[(3)] a linear morphism $\mathcal{F}(t,e):\mathcal{F}(e)\rightarrow \mathcal{F}(t(e))$ for each edge $e\in E\Gamma$; and
    \item[(4)] a linear morphism $\mathcal{F}(h,e):\mathcal{F}(e)\rightarrow \mathcal{F}(h(e))$ for each edge $e\in E\Gamma$.
\end{itemize}
\end{definition}

 Let $\gs$ be the scwol constructed from $\Gamma$ as in \ref{scwol construction}. In this situation, it is easy to see that our definition of a sheaf from Definition \ref{D: sheaf} is equivalent to Friedman's. 
\vspace{.1in}

We now give the definition of complexes of groups. Complexes of groups are natural generalizations of graphs of groups from Bass--Serre theory. Our definition is derived from \cite[III.$\scat$  2.1 p.535]{MR1744486}.

\begin{definition}\label{d: complexes of groups}
A \textbf{complex of groups} $\Lambda = \big(\gs$, $G_{\sigma}$, $\psi_a$, $g_{a,b}\big)$ is,
\begin{itemize}
    \item A scwol $\gs$;
    \item For each $\sigma\in \Obj(\gs)$ a group $G_{\sigma}$ called the \textbf{local group at } $\sigma$;
    \item For each $a\in\Arr(\gs)$ an injective homomorphism $\psi_a:G_{t(a)}\rightarrow G_{i(a)}$; and 
    \item For each pair of composable edges $(a,b)\in E^{(2)}(\gs)$, a twisting element $g_{a,b}\in G_{t(a)}$ which satisfies, 
    \begin{itemize}
        \item $\mathrm{Ad}(g_{a,b})\psi_{ab} = \psi_a\psi_b$
        
        where $\mathrm{Ad}(g_{a,b})$ is the conjugation by $g_{a,b}$ in $G_{t(a)}$,
        
        \item For each triple $(a,b,c)\in E^{(3)}(\gs)$ of composable elements we have the cocycle condition:
        $$
        \psi_a\big(g_{b,c}\big)g_{a,bc}=g_{a,b}g_{ab,c}
        $$
    \end{itemize}
\end{itemize}
\end{definition}
One way to construct a complex of groups is via a \textit{group action} on a swol, see \cite[III.$\scat$ 1.11 p.528]{MR1744486}.  
\vspace{.2in}

In \cite[III.$\scat$ 2.8 p.538]{MR1744486} Bridson and Haefliger construct a category $\bh$ from a complex of groups $\cg$. We give the definition of $\bh$, the \textit{B.H category}, below. 

\begin{definition}
Let  $\Lambda =\big(\gs, G_{\sigma}, \psi_a, g_{a,b}\big)$ be a complex of groups over a scwol $\gs$. \textbf{The B.H. category of} $\Lambda$ is the category $C\Lambda$ given by,
\begin{itemize}
    \item $\Obj\big(C\Lambda\big) = \Obj(\gs)$;
    \item $\Arr\big(C\Lambda\big) = \big\lbrace (g,\alpha): \alpha\in \mathrm{Mor}\big(\gs\big)$, $g\in G_{t(\alpha)}\big\rbrace$ with the adjacency relations;
    \begin{itemize}
        \item $i(g,\alpha) = i(\alpha)$
        \item $t(g, \alpha) =t(\alpha)$
    \end{itemize}
    For arrows $(\alpha, g)$, $(\beta, h)\in \Arr(C\Lambda)$, $i(\alpha) = t(\beta)$ we define their composition,
    $$
    (g, \alpha)(h, \beta) = (g\psi_{\alpha}(h)g_{a,b}, \alpha\beta), \hspace{.1in}g_{a,b}\in G_{t(a)}
    $$
\end{itemize}
\end{definition}
Let $\pi_1(\cg)$ be the fundamental group of a complex of groups as in \cite[II.$\scat$ 3.7 p.549]{MR1744486} and $\pi_1(\bh)$ the fundamental group of the category as defined in \cite[III.$\scat$ A.13 587]{MR1744486}. We get,
\begin{theorem}\textup{\cite[III.$\scat$ A.12 p.578]{MR1744486}}
Let $\cg$ be a complex of groups and $\bh$ the B.H. category. Then,
$$
\pi_1(\cg) \cong \pi_1(\bh)
$$
\end{theorem}
\begin{definition}
Let $\gs$ and $\gs'$ be scwols. 
\begin{itemize}
    \item A \textbf{morphism of scwols} $f:\gs\rightarrow \gs'$ is a functor.

\item A \textbf{non-degenerate morphism} of scwols is a functor such that for each $v\in\Obj(\gs)$ there is a bijection of sets, 
$$
\big\lbrace a\in \Arr(\gs): i(a) = v\big\rbrace \leftrightarrow \big\lbrace a'\in \Arr(\gs'): i(a) = f(v)\big\rbrace.
$$
\end{itemize}
\end{definition}
We define a morphism of complexes of groups over a morphism of scwols. 

\begin{definition}\textup{\cite[III.$\scat$.2.4, p. 536]{MR1744486}}\label{d: morph of complexes of groups}
Let $\cg = (\gs, G_\sigma, \psi_a, g_{a,b})$ and $\cg' =  (\gs', G_\sigma', \psi_a', g_{a,b}')$ be a complexes of groups and $f:\gs\rightarrow \gs'$ a functor. A (possibly degenerate) \textbf{morphism of complexes of groups} $\phi:(\phi_{\sigma},\phi(a)):\cg\rightarrow \cg'$ over $f$ consists of:
\begin{itemize}
    \item a homomorphism $\phi_{\sigma}:G_{\sigma}\rightarrow G_{f(\sigma)}$ for each $\sigma\in \Obj(\gs)$; and
    \item an element $\phi(a)\in G_{t(f(a))}$ for each $a\in \Arr(\gs)$ such that,
    \begin{itemize}
        \item $\mathrm{Ad}(\phi(a))\psi_{f(a)}(\phi_{i(a)}) = \phi_{t(a)}\psi_a$,
        
        where $\mathrm{Ad}(\phi(a))$ is the conjugation by $\phi(a)$ in $G_{t(f(a)))}$,
        
        and for each pairs of composable arrows $(a,b)\in E^{(2)}(\gs)$,
        \item $\phi_{t(a)}(g_{a,b})\phi(ab) =\phi(a)\psi_{f(a)}(\phi(b))g_{f(a),f(b)}$.
    \end{itemize}
\end{itemize}
\end{definition}
When each of the homomorphisms are injective we get,
\begin{definition}\label{d: inj local complex groups}
Let $\phi = (\phi_{\sigma}, \phi(a)):\cg\rightarrow\cg'$ be a morphism of complexes of groups. If each homorphism $\phi_{\sigma}$ is injective we say that $\phi$ is \textbf{injective on local groups}. 
\end{definition}
A B.H. category is constructed to carry the same data as the complex of groups it represents. It follows that a morphism of complexes of groups induces a functor between their B.H. categories. This is the content of the next definition. 
\begin{definition}\textup{\cite[III.$\scat$ 2.8 p.538]{MR1744486}} \label{induced functor}
Let $\cg$, $\cg'$ be complexes of groups and $\bh$, $\bh'$ their B.H. categories. A morphism of complexes of groups $\phi:\cg\rightarrow \cg'$ induces the \textbf{induced functor} between their B.H. categories $\widehat{\phi}:\bh\rightarrow\bh'$ and a functor between B.H. categories $\widehat{\phi}:\bh\rightarrow\bh'$ induces the \textbf{induced morphism} between the complexes of groups they represent $\phi:\cg\rightarrow\cg'$.
\end{definition}

The importance of our next lemma is to provide a translation between complexes of groups and their B.H. categories. Each morphism of complexes of groups induces a functor between their B.H. categories that carries the same data and conversely each functor between B.H. categories induces a morphism of complexes of groups that carries the same data. 

\begin{lemma}[cf, {\cite[III.$\scat$ 2.8 p.538]{MR1744486}}] \label{onetwo}
\label{mor to func}
The procedure of constructing B.H. categories from complexes of groups induces a bijection between morphism of complexes of groups and functors between their B.H. categories. 
\end{lemma}
We introduce the trivial scwol which aids us in subsequent definitions. 
\begin{definition}
The \textbf{trivial scwol} $\mathfrak{P}$ is a scwol with one object and no arrows. 
\end{definition}
Given a group $G$ we construct the following complex of groups,
\begin{definition}
Let $G$ be a group. The \textbf{complex of groups constructed from the group} $G$ is the (unique!) complex of groups $\Theta(G)$ over $\mathfrak{P}$ with local group $G$.
\end{definition}

Let $\cg$ be a complex of groups, $G$ a group and $\phi:\cg \rightarrow \Theta(G)$ a morphism. Due to the lack of composable edges in $\Theta(G)$ the definition of a morphism becomes (cf. \cite[III.$\scat$ 2.5 p.537]{MR1744486})

\begin{definition} \label{complex of groups morphism}
A \textbf{morphism} $\phi = (\phi_{\sigma}, \phi(a))$ from a complex of groups $\Lambda$ to $\Theta(G)$ consists of a homorphism $\phi_{\sigma}:G_{\sigma}\rightarrow G$ for each $\sigma\in \Obj(\Lambda)$ and an element $\phi(a)\in G$ for each $a\in \Arr(\Lambda)$ such that 
$$
\phi_{t(a)}\psi_a = \mathrm{Ad}\big(\phi(a)\big)\phi_{i(a)} \hspace{.1in}\text{ and }\hspace{.1in} \phi_{t(a)}(g_{a,b})\phi(ab) = \phi(a)\phi(b)
$$
\end{definition}

Let $\cg$ be a complex of groups over a scwol $\gs$, $\sigma\in \Obj(\gs)$, $\bh$ the B.H. category and $\sigma'\in \Obj(\bh)$ the object constructed from $\sigma$. Each local group $G_{\sigma}$ of $\cg$ corresponds to a full subcategory over each $\sigma'\in \bh$. We make this precise.  

\begin{definition}
Let $\cg$ be a complex of groups over a scwol $\gs$ and $\bh$ the B.H category with $\sigma\in\Obj(\bh)$. Then $S_{\sigma}$ is the full subcategory of $\bh$ over $\sigma$.
\end{definition}

For any group $G$ we can construct the following category,

\begin{definition} \label{groupCat}
Let $G$ be a group. The \textbf{category constructed from the group} $\widehat{G}$ is defined by taking,
\begin{itemize}
    \item $\Obj(\widehat{G}) = \tau_{G}$
    \item $\Arr(\widehat{G}) = \big\lbrace \alpha_G(g) : g\in G\big\rbrace$ 
    \item With composition defined,
    $$
    \alpha_G(g)\alpha_G(h) = \alpha_G(gh)
    $$
\end{itemize}
\end{definition}
Observe that for a group $G$ we obtain, 
\begin{lemma} \label{bh to hat}
Let $G$ be a group, $\Theta(G)$ the complex of groups construct from the group, $\mathcal{C}\Theta(G)$ the B.H. category and $\widehat{G}$ the category constructed from the group. Then $\mathcal{C}\Theta(G)$ and $\widehat{G}$ are isomorphic.
\end{lemma} 
\begin{definition}
Let $G$ be a group, $\Theta(G)$ the complex of groups constructed from $G$, $\mathcal{C}\Theta(G)$ the B.H. category and $\widehat{G}$ the category constructed from $G$. The \textbf{B.H. to hat map} is the functor that realizes the isomorphism of Lemma \ref{bh to hat} and is given by,
$$
\psi_{G}:\mathcal{C}\Theta(G) \rightarrow \widehat{G}
$$
such that,
\begin{itemize}
    \item for $\sigma \in \Obj(\mathcal{C}\Theta(G))$, $\psi_G(\sigma) = \tau_G$; and
    \item for $(g,\mathbbm{1}_{\sigma})\in \Arr(\mathcal{C}\Theta(G))$, $\psi_G((g,\mathbbm{1}_{\sigma})) = \alpha_G(g)$.
\end{itemize}
\end{definition}

Let $\cg$ be a complex of groups over a scwol $\gs$, $G_{\sigma}$ a local group of $\cg$ with $\sigma\in \Obj(\gs)$, $\bh$ the B.H. category and $S_{\sigma}\subseteq \bh$ the full subcategory over $\sigma\in \Obj(\bh)$. We get,

\begin{lemma} \label{L: taut natural iso}
For $\sigma\in \Obj(\bh)$,
$$
S_{\sigma} \cong \scat\Theta(\mathrm{Aut}_{\bh}(\sigma)) \cong \scat\Theta(G_{\sigma}) \cong \widehat{G}_{\sigma}
$$
\end{lemma}

From Lemma \ref{L: taut natural iso} we define,
\begin{definition}\label{full subcat to BH local group}
The \textbf{full subcategory to B.H. category functor} is,
$$
\Upsilon_{S_{\sigma}}:S_{\sigma}\rightarrow \mathcal{C}\Theta(G_{\sigma})
$$
\end{definition}

Let $\cg$ be a complex of groups over a scwol $\gs$, $G$ a group. for $\sigma\in \Obj(\gs)$ and $a\in \Arr(\gs)$ the map $\phi = (\phi_{\sigma}, \phi(a)):\cg\rightarrow \Theta(G)$ is a morphism of complexes of groups with local group homomorphism $\phi_{\sigma}:G_{\sigma}\rightarrow G$. By Lemma \ref{induced functor}, $\phi$ induces the functor $\widehat{\phi}:\bh\rightarrow\mathcal{C}\Theta(G)$ and the map $\phi_{\sigma}$ induces the functor $\widehat{\phi}_{\sigma}:\mathcal{C}\Theta(G_{\sigma})\rightarrow \mathcal{C}\Theta(G)$. Let $\psi_G: \mathcal{C}\Theta(G)\rightarrow \widehat{G}$ be the B.H. to hat map from Lemma \ref{bh to hat} and $\Upsilon_{S_{\sigma}}:S_{\sigma}\rightarrow \mathcal{C}\Theta(G_{\sigma})$ the full subcategory to B.H. category functor from Definition \ref{full subcat to BH local group}. We define the following two functors, 

\begin{definition}\label{cathatmap}
The \textbf{cat to hat map} is,
\begin{itemize}
    \item[] 
    $$\chat = \widehat{\psi}_{G}\circ \widehat{\phi}:\bh\rightarrow \mathcal{C}\Theta(G) \rightarrow\widehat{G}
    $$
\end{itemize}
For $\sigma\in \Obj(\bh)$ the \textbf{local cat to hat map} is,
\begin{itemize}
    \item[] 
    $$
    \lchat = \widehat{\psi}_{G}\circ \widehat{\phi}_{\sigma}\circ \widehat{\Upsilon}_{S_{\sigma}}:S_{\sigma}\rightarrow \mathcal{C}\Theta(G_{\sigma})\rightarrow\widehat{G}_{\sigma} \rightarrow \widehat{G}
    $$
\end{itemize}
\end{definition}
Observe that by restricting $\bh$ to the full subcategory $S_{\sigma}$ over $\sigma\in \Obj(\bh)$ we get,

$$
\chat|_{S_{\sigma}} = \lchat 
$$
\vspace{.2in}

The definition of a group $G$ acting of a scwol $\mathcal{X}$ is stronger than the group of automorphisms $\mathrm{Aut}(\mathcal{X})$ of $\mathcal{X}$ so that the quotient category $G\setminus\mathcal{X}$ is always a scwol. See \cite[III.$\scat$ 1.11 p.528]{MR1744486} for the definition of a group $G$ acting on a scwol $\mathcal{X}$. The most natural way to build a complex of groups $\cg = (\mathcal{Y}, G_{\sigma},\psi_a, g_{a,b})$ is to construct $\cg$ over a quotient scwol $\mathcal{Y} = G\setminus \mathcal{X}$. Here the local groups $G_{\sigma}$ of $\cg$ are stabilizers in $G$ of objects in $\mathcal{X}$. Furthermore, when a complex of groups $\cg$ is the result of a global quotient we say that $\cg$ is developable. We give the precise definition below.  

\begin{definition}
A complex of groups $\Lambda = (\mathcal{Y}, G_{\sigma},\psi_a, g_{a,b})$ is called \textbf{developable} if it is isomorphic to a complex of groups that is associated to an action of a group $G$ on a scwol $\mathcal{X}$ with $\mathcal{Y} = G\setminus \mathcal{X}$ the quotient scwol. That is, $\mathcal{Y}$ is the scwol for $\Lambda$.
\end{definition}
Not all complexes of groups are developable. See Example (5) \cite[II.12.17 (5) p.378]{MR1744486} for an enxample of a non-developable complex of groups.

In \cite[III.$\scat$ 2.15 p.544]{MR1744486} Bridson and Haefliger give the following characterization of developability of a complex of groups.

\begin{theorem}\label{key theorem}
A complex of groups $\cg$ is developable if and only if there exists a group $G$ and a morphism $\phi:\cg \rightarrow \Theta(G)$ which is injective on local groups.
\end{theorem}
We make a definition so that we may phrase developability of a complex of groups in terms of the group $G$.
\begin{definition}
Let $\cg$ be a developable complex of groups and $G$ a group such that there exists a morphism $\phi:\cg \rightarrow \Theta(G)$ that is injective on local groups. Then we say that $G$ \textbf{witnesses the developability} of $\cg$.  
\end{definition}

In case of an ordinary complex take $G = \lbrace 1 \rbrace$ (or any other groups like $\GL(V)$, with all homomorphisms trivial). So although the morhpism $\phi:\pi_1(\cg) \rightarrow G$ is injective on local groups, $\phi$ needn't be injecitive. 

In the next section we define sheaves. 

\section{Sheaves}
We restate our main definition,
\Sheaf*

The B.H. category is constructed to carry the data of a complex of groups. When $\cg$ is a complex of groups a \textit{sheaf over }$\cg$ will mean a sheaf over $\bh$. 
\vspace{.2in}

We define sheaf constructions that we use later in this paper. 
\begin{definition}
Let $\mathcal{C}$, $\mathcal{C}'$ be categories and $\phi:\mathcal{C}\rightarrow\mathcal{C}'$ a functor. Let $\mathcal{F}$ be a sheaf on $\mathcal{C}'$. We define the \textbf{pullback sheaf } $\phi^*\mathcal{F}$ on $\scat$ to be. 
$$
\phi^*\mathcal{F} = \mathcal{F}\circ \phi
$$
\end{definition} 
\begin{definition}
Let $\mathcal{F}$ be a sheaf over a category $\scat$ and let $n\in \mathbb{N}$. If for each $\sigma\in \Obj(\scat)$ we have $\dim(\mathcal{F}(\sigma)) = n$ then $\mathcal{F}$ has \textbf{constant rank}.
\end{definition}

Let $\cg$ be a complex of groups and $\bh$ the B.H. category. Since the full subcategory $S_{\sigma}$ over $\sigma\in \Obj(\bh)$ carries data of the local group $G_{\sigma}$ of $\cg$ we obtain a sheaf version of Definition \ref{d: inj local complex groups}.
\begin{definition}
Let $\scat$ be a category and $\mathcal{F}$ a sheaf over $\scat$. If for each $\sigma\in \Obj(\scat)$ the restriction $\mathcal{F}|_{S_{\sigma}}:S_{\sigma} \rightarrow \mathbb{K}\mathrm{-Vect}$ is injective then $\mathcal{F}$ is said to be \textbf{injective on local groups}.
\end{definition}
We summarize the previous sheaf properties into a single definition along with one more property.
\begin{definition}
Let $\mathcal{F}$ be a sheaf. We say that $\mathcal{F}$ satisfies the \textbf{dev properties} if,
\begin{itemize}
    \item $\mathcal{F}$ has constant rank;
    \item if for each $\lambda\in \mathrm{Mor}(\bh)$, $\mathcal{F}(\lambda)$ is invertible; and if
    \item $\mathcal{F}$ is injective on local groups.
\end{itemize}
\end{definition}

Observe that when we construct a category from a group $G$ we obtain,

\begin{example}
Let $G$ be a group, $\widehat{G}$ the category constructed from $G$ as in Definition \ref{groupCat}. Then a sheaf $\mathcal{F}$ over $\widehat{{G}}$ is a linear representation of $G$.  
\end{example}

Let $\GL(V)$ be the general linear group over the vector space $V$ and let $\widehat{\GL(V)}$ be the category constructed from $\GL(V)$. 
\begin{definition}
The \textbf{tautological sheaf} is the sheaf $\mathcal{Z}_{taut}^V:\widehat{\GL(V)}\rightarrow \mathbb{K}\mathrm{-Vect}$ so that,
\begin{itemize}
    \item $\mathcal{Z}_{taut}^V\big(\tau_{\GL(V)}\big) = V$ for $\tau_{\GL(V)}\in \Obj(\widehat{\GL(V)})$; and
    \item $\mathcal{Z}^V_{taut}\big(\alpha_{\GL(V)}(\rho)\big) = \rho$ for $\alpha_{\GL(V)}(\rho)\in\Arr(\widehat{\GL(V)})$ 
\end{itemize}. 
\end{definition}

Recall that $\GL(V)\subset \Arr(\mathbb{K}\mathrm{-Vect})$ so $\alpha_{\GL(V)}(\rho)\in \Arr(\widehat{\GL(V)})$ where $\rho\in\GL(V)$. When we pass $\alpha_{\GL(V)}(\rho)$ through the tautological sheaf $\mathcal{Z}_{taut}^V$ we recover the group element $\mathcal{Z}^V_{taut}\big(\alpha_{\GL(V)}(\rho)\big) = \rho\in \GL(V)$.


\begin{lemma}\label{onedirection}
Let $\cg$ be a developable complex of groups such that $\GL(V)$ witnesses the developability of $\cg$, let $\phi:\cg\rightarrow \Theta(\GL(V))$ be a morphism that is injective on local groups and $\chat$ the cat to hat map of Definition \ref{cathatmap}. Then $\chat\circ \mathcal{Z}^V_{taut}$ is a sheaf that satisfies the dev-properties.
\end{lemma}
\begin{proof}
The complex of groups $\Theta(\GL(V))$ is constructed over the trivial scowl $\mathfrak{P}$. The trivial scwol has one object and no arrows so the rank of $\widehat{\Psi}^{\ast}\circ \mathcal{Z}^V_{taut}$ is constant and each arrow represents an invertible group element of $\GL(V)$ so all morphisms are invertible. We proceed to prove injectivity on local groups.
\vspace{.1in}

Let $\gs$ be the scwol that $\cg$ is over, $\sigma\in \Obj(\gs)$, $a\in \Arr(\gs)$ and recall that for $\phi = (\phi_{\sigma}, \phi(a))$ the map $\phi_{\sigma}:G_{\sigma}\rightarrow \GL(V)$ is a homomorphism and $\lchat:S_{\sigma}\rightarrow \widehat{\GL(V)}$ is the local cat to hat map from Definition \ref{cathatmap}.

Since $\chat|_{S_{\sigma}} = \lchat$, it suffices to show that for all $\sigma\in \Obj(\bh)$ each $\lchat^{\ast}\circ\mathcal{Z}^V_{taut}$ is injective. 

For $\sigma\in \Obj(\bh)$ and $(k,\mathbbm{1}_{\sigma})\in \Arr(S_{\sigma})$ we have the following chain of equalities,

\begin{itemize}
    \item[(1)] \hspace{.4in} $\phi_{\sigma}(k) = \mathcal{Z}^V_{taut}\big(\alpha_{\GL(V)}(\phi_{\sigma}(k))\big) = \lchat^{\ast}\circ\mathcal{Z}^V_{taut}\big((k,\mathbbm{1}_{\sigma})\big)$.
\end{itemize}

Let $(g,\mathbbm{1}_{\sigma}), (h,\mathbbm{1}_{\sigma})\in \Arr(S_{\sigma})$ so $\lchat^{\ast}\circ\mathcal{Z}^V_{taut}\big((g,\mathbbm{1}_{\sigma})\big) = \lchat^{\ast}\circ\mathcal{Z}^V_{taut}\big((h,\mathbbm{1}_{\sigma})\big)$.  From equation $(1)$ above we get,
\begin{align*}
\phi_{\sigma}(g) &= \mathcal{Z}^V_{taut}\big(\alpha_{\GL(V)}(\phi_{\sigma}(g))\big)\\
    &= \lchat^{\ast}\circ\mathcal{Z}^V_{taut}\big((g,\mathbbm{1}_{\sigma})\big)\\
    &=\lchat^{\ast}\circ\mathcal{Z}^V_{taut}\big((h,\mathbbm{1}_{\sigma})\big)\\
    &= \mathcal{Z}^V_{taut}\big(\alpha_{\GL(V)}(\phi_{\sigma}(h))\big) \\
    &= \phi_{\sigma}(h) 
\end{align*}
We see that $\phi_{\sigma}(g) = \phi_{\sigma}(h)$. Since $\phi$ is injective on local groups then $\phi_{\sigma}$ is injective. It follows that $g = h$ and hence $\lchat^{\ast}\circ \mathcal{Z}^V_{taut}$ is injective. Since the choice of $\sigma \in \Obj(\bh)$ was arbitrary it follows that $\chat^{\ast}\circ \mathcal{Z}^V_{taut}$ is injective on local groups.
\end{proof}
Proving the other direction of Theorem \ref{main} will be the content of the next section.

\section{Proof of Theorem \ref{main}.}
We restate and prove Theorem \ref{main}. The proof technique will be familiar to constructing the fundamental group of a graph via a maximal tree. We take the maximal tree $T$ of a scwol $\gs$ that a complex of groups $\cg$ is defined over and witness how the data interacts when observed through a sheaf. In Subsection \ref{scwol def} we give the technical definitions concerning the scwol $\gs$. In Subsection \ref{les functors} we define functors that will then be to define \textit{sheaf paths} and a morphism of complexes of groups in Subsection \ref{sheafpaths} which allows us to prove Theorem \ref{main} in Subsection \ref{proofA}.  

\subsection{Edge paths and a Tree} \label{scwol def}
In this subsection we define edge paths of a scwol $\mathcal{X}$ and give the definition of a maximal tree $\T$ of $\mathcal{X}$. We use $\T$ in Subsection \ref{sheafpaths} to construct a morphism of complexes of groups. Many of these definitions can be found in or are inspired from \cite[III.$\scat$ 1.6 p.526]{MR1744486}. 
\vspace{.2in}

Let $\gs$ be a scwol and let $E^{\pm}(\gs)$ be the set of symbols $a^+$ and $a^-$, where $a\in \Arr(\gs)$.

\begin{definition}
The element $e\in E^{\pm}$ is an \textbf{oriented edge} of $\gs$. If $e = a^+$ then $i(e) = t(a)$, $t(e) = i(a)$, and if $e = a^-$ we have $i(e) = i(a)$ and $t(e) = t(a)$.  
\end{definition}
\begin{definition}
Let $v,w\in \Obj(\gs)$. An \textbf{edge path} of $\gs$ joining $v$ to $w$ is a sequence of elements in $E^{\pm}(\gs)$, $c = (a_1^{\epsilon},\cdots,a_n^{\epsilon_n})$ with $a_k^{\epsilon_k}\in E^{\pm}(\gs)$, $a_k\in \Arr(\gs)$, $\epsilon_k\in \big\lbrace -1, 1\big\rbrace$, such that for all $1< k < n$ we have $i(a_k^{\epsilon_k}) = t(a_{k-1}^{\epsilon_{k-1}})$ and $i(c) = v$, $t(c) = w$. 

Given $\sigma \in \Obj(\gs)$ the \textbf{constant path} over $\sigma$ is realized by $\mathrm{id}_{\sigma}$.
\end{definition}

\begin{definition}
Let $c = (a_1^{\epsilon},\cdots,a_n^{\epsilon_n})$ be an edge path. We say that \textbf{backtracking} occurs in $c$ if for any sequential edges $a_j,a_{j+1}$ in $c$ we have $a_j = a_{j+1}$ and $\epsilon_j = \epsilon^{-1}_{j+1}$ and \textbf{reduce} $c$ by removing $a_j$ and $a_{j+1}$ from $c$ and forming a new edge path from the remaining edges. If $c$ exhibits no backtracking then $c$ is said to be \textbf{fully reduced}.   
\end{definition} 
\begin{definition}
Let $c = (a_1^{\epsilon},\cdots,a_n^{\epsilon_n})$ be an edge path in $\gs$. We say that $c$ is a \textbf{loop} if $i(c) = t(c)$ and say that $c$ is a \textbf{circuit} if $c$ is also embedded. 
\end{definition}
\begin{definition}
A \textbf{tree} $\T$ is a subset of objects $\Obj(\T)\subseteq \Obj(\gs)$ and a subset of arrows $\Arr(\T) \subseteq \Arr(\gs)$ so that for all $a\in \Arr(\T)$ we have $i(a),t(a)\in \Obj(\T)$, $\T$ is connected, and $\T$ has no circuits. We say that $\T$ is a \textbf{maximal tree} if $\Obj(\T)=\Obj(\gs)$. 
\end{definition}
 \begin{definition}\label{d: utp}
 Let $T\subseteq \gs$ be a tree with $v,w\in \Obj(T)$ and let $p_T(v,w)$ be the unique reduced edge path between $v$ and $w$ in $T$. Then $p_T(v,w)$ is \textbf{unique tree path} between $v$ and $w$.
 \end{definition}
 
 We next define and explain the functors that we use to prove Theorem \ref{main}. 
 \subsection{The Functors}\label{les functors}
 Let $V$, $W$ be vector spaces and $L\in \Arr(\mathbb{K}\mathrm{-Vect})$, $L:V\rightarrow W$ an isomorphism of vector spaces. We construct the categories $\widehat{\GL(V)}$, $\widehat{\GL(W)}$. Then $L$ induces the following functor, 
 \begin{definition}
 The $\widehat{\textbf{L}}$\textbf{-functor} $ \widehat{L}:\widehat{\GL(V)}\rightarrow \widehat{\GL(W)}$ is,
defined by,
\begin{itemize}
    \item  $\L\big(\tau_{\GL(V)}\big) = \tau_{\GL(W)})$; and
    \item $\L\big(\alpha_{\GL(V)}(\rho) \big) = \alpha_{\GL(W)}\big(L\rho L^{-1}\big)$.
\end{itemize}

 \end{definition}
 
 The $\L$-functor induces the following natural transformation,
 \begin{definition}
 The \textbf{L-transformation} $\eta^L_{\GL(V)}:\mathcal{Z}^V_{taut}\rightarrow \L^{\ast}\mathcal{Z}^W_{taut}$ is the map $\eta^L_{\GL(V)} = L$.
 \end{definition}
 It is a quick exercise to show $\widehat{L}$ is indeed a functor. Whats more, it is not hard to show that $\eta^L_{\GL(V)}$ is a \textit{natural isomorphism}. This follows from $L$ being invertible since a natural isomorphism is a natural transformation whose component maps are isomorphisms.

 Let $\cg$ be a complex of groups and $\bh$ the B.H. category. For $\sigma\in \Obj(\bh)$ let $S_{\sigma}\subseteq \bh$ be the full subcategory over $\sigma$ with inclusion map $i_{\sigma}:S_{\sigma}\rightarrow \bh$. Recall that $S_{\sigma} \cong \widehat{G_{\sigma}}$. Let $\mathcal{F}$ be a sheaf over $\cg$. .
 \begin{definition}
 The $\widehat{\textbf{I}_{\sigma}}$\textbf{-functor} $\widehat{I_{\sigma}}:S_\sigma \rightarrow \widehat{\GL(\mathcal{F}(\sigma))}$ is defined by,
 \begin{itemize}
     \item $\widehat{I_{\sigma}}(\sigma) = \tau_{\GL(\mathcal{F}(\sigma))}$; and
     \item $\widehat{I_{\sigma}}((g,\mathbbm{1}_{\sigma})) = \alpha_{\GL(\mathcal{F}(\sigma))}(\mathcal{F}((g,\mathbbm{1}_{\sigma})))$.
 \end{itemize}
 \end{definition}
 Furthermore, let $\mathcal{Z}_{taut}^{\mathcal{F}(\sigma)}$ be the tautological sheaf. It is not difficult to see that that $i^{\ast}_{\sigma}\mathcal{F}$ and $\widehat{I}_{\sigma}^{\ast}\mathcal{Z}_{taut}^{\mathcal{F}(\sigma)}$ are natural isomorphisms and furthermore, 
 $$
 i^{\ast}_{\sigma}\mathcal{F} = \widehat{I}_{\sigma}^{\ast}\mathcal{Z}_{taut}^{\mathcal{F}(\sigma)}.
 $$
 This is summarized in the following diagram,
 $$
 \xymatrix{
 &S_{\sigma} \ar[dr]^{i_{\sigma}}\ar[dl]_{\widehat{I_{\sigma}}} &\\
 \widehat{\GL(\mathcal{F}(\sigma))}\ar[dr]^{\mathcal{Z}_{taut}^{\mathcal{F}(\sigma)}} & &\bh \ar[dl]_{\mathcal{F}}\\
 &\mathbb{K}\mathrm{-Vect} & 
 }
 $$
 \subsection{Sheaf Paths and a Morphism of Complexes of Groups}\label{sheafpaths}
  We utilize the functors from the previous subsection to define paths in the sheaf $\mathcal{F}$ over edge paths in the scwol. We use sheaf paths to define a functor $\widehat{\phi}$ which induces a morphism of complexes of groups. 
  
  The following is an example of an $\widehat{L}$-functor.
 \begin{definition}
 Let $\cg$ be a complex of groups over a scwol $\gs$ with maximal tree $T\subseteq \gs$ and $p_T(v,w) = (a_1^{\epsilon_1}, \cdots, a_n^{\epsilon_n})$ a unique tree path in $T$. Let $\mathcal{F}$ be a sheaf over $\cg$ such that $\mathcal{F}(v) = V$. A \textbf{sheaf path towards} $w$ $M_w:V\rightarrow \mathcal{F}(w)$ is,
 $$
 M_{w} = \prod_{j = 0}^{n-1}\mathcal{F}\big((\mathbf{1}, a_{n-j})\big)^{\epsilon_{n-j}}
 $$
 \end{definition}
 Each sheaf path towards $w$, $M_w$ induces a natural isomorphism $\eta_{\GL(V)}^{M_w}:\mathcal{Z}_{taut}^V\rightarrow \widehat{M}^{\ast}_w\mathcal{Z}_{taut}^{\mathcal{F}(w)}$.

 We now define the \textit{local group functor}.
 
 \begin{definition}\label{lgf}
 Let $\cg = (\gs, G_{\sigma}, \psi_a, g_{a,b})$ be a complex of groups, $\mathcal{F}$ a sheaf over $\cg$, $V$ a vector space, $M_w:V\rightarrow \mathcal{F}(w)$ a sheaf path towards $w$. Let $\sigma\in \Obj(\gs)$. The \textbf{local group functor} $\widehat{\phi_{\sigma}}:\scat\Theta(G_{\sigma})\rightarrow \widehat{\GL(V)}$ is given by,
 \begin{itemize}
     \item $\widehat{\phi_{\sigma}}(\tau_{G_{\sigma}}) = \tau_{\GL(V)}$ for $\tau_{G_{\sigma}}\in \Obj(\scat\Theta(G_{\sigma}))$; and
     \item $\widehat{\phi_{\sigma}}(\alpha_{G_{\sigma}}(g)) = \alpha_{\GL(V)}\big(M^{-1}_{\sigma}\mathcal{F}(g, \mathbbm{1}_{\sigma})M_{\sigma}\big)$ for  $\alpha_{G_{\sigma}}(g) \in \Arr(\scat\Theta(G_{\sigma}))$.
 \end{itemize}
 \end{definition}
 By composing natural isomorphisms it is not hard to see,
 \begin{lemma}
 If $\mathcal{F}$ satisfies the dev-properties then for all $\sigma\in \Obj(\gs)$, $\widehat{\phi_{\sigma}}$ is injective.
 \end{lemma}
 By focusing on group elements we get,
 \begin{lemma} \label{local group hom}
 If $\mathcal{F}$ satisfies the dev-properties then for all $\sigma\in \Obj(\gs)$, $\widehat{\phi_{\sigma}}$ induces an injective group homomorphism $\phi_{\sigma}:G_{\sigma}\rightarrow \GL(V)$ 
 \end{lemma}
We define the following element of the group $\GL(V)$,
\begin{definition}\label{gel}
Let $a\in \Arr(\gs)$, $g\in G_{t(a)}$, $\mathcal{F}(g,a):\mathcal{F}(i(a))\rightarrow \mathcal{F}(t(a))$, and let $M_{t(a)}:V\rightarrow \mathcal{F}(i(a))$ and $M^{-1}_{t(a)}:\mathcal{F}(t(a))\rightarrow V$ be sheaf paths. We define $\phi_g(a)\in \GL(V)$ as,
$$
\phi_g(a) = M^{-1}_{t(a)}\circ\mathcal{F}(g,a)\circ M_{i(a)}
$$
\end{definition}

\begin{definition}\label{keydef}
Let $T\subseteq \gs$ be a maximal tree. We define the map $\phi = (\phi_{\sigma}, \phi_{\mathbf{1}}(a)):\cg\rightarrow \Theta\big(\GL(V)\big)$ by,
\begin{itemize}
    \item for each $\sigma\in \Obj(T)$, let $\phi_{\sigma}:G_{\sigma}\rightarrow\GL(V)$ be the homomorphism constructed in Definition \ref{lgf}; and
    \item for $a\in \Arr(\gs)$, let $\phi_{\mathbf{1}}(a)\in \GL(V)$ be the construction from Definition \ref{gel}.
\end{itemize}
\end{definition}
We state the key lemma used to prove the other direction of Theorem \ref{main}.
\begin{lemma}[Key Lemma]\label{keylemma}
Let $\phi = (\phi_{\sigma}, \phi_{\mathbf{1}}(a))$ be the map constructed in Definition \ref{keydef}. Then $\phi$ is a morphism of complexes of groups. If furthermore $\mathcal{F}$ satisfies the dev-properties then $\phi$ is injective on local groups.
\end{lemma}
Before we prove the Key Lemma we state and prove two easy technical lemmas concerning a constant rank sheaf whose linear morphisms are all invertible. These results rely on the construction of the B.H category and on the fact that functors preserve composition of arrows.

\begin{lemma}\label{L: comp 1.a}
Let $\Lambda = (\gs, G_{\sigma}, \psi_a, g_{a,b})$ be a complex of groups, $C\Lambda$ the B.H. category and $\mathcal{F}$ a constant rank sheaf over $\cg$ whose linear morphisms are all invertible. Then for $(\idGrp,a),(g,\idCat{i(a)}),(\psi_{a}(g),\idCat{t(a)}) \in \Arr(C\Lambda)$ we get the relation, 
$$
\mathcal{F}((g,\idCat{i(a)})) = \mathcal{F}\big((\idGrp,a)\big)^{-1}\circ\mathcal{F}\big((\psi_a(g),\idCat{t(a)})\big)\circ\mathcal{F}\big((\idGrp,a)\big) 
$$
\end{lemma}
\begin{proof}
The arrows fit into the commutative diagram,
$$
\xymatrix{
\tau_{i(a)} \ar[r]^{(\idGrp,a)} \ar[d]_{(g,\idCat{\tau_{i(a)}})} & \ar[d]^{(\psi_a(g),\mathbbm{1}_{\tau_{t(a)}})} \tau_{t(a)}\\
\tau_{i(a)} \ar[r]^{(\idGrp,a)} & \tau_{t(a)}
}
$$
The diagram yields the relation, 
$$
 (\idGrp,a)(g,\idCat{\tau_{i(a)}}) = (\psi_a(g),\idCat{\tau_{t(a)}})(\idGrp,a)
$$
Functors preserve composition of arrows. Since all linear morphisms of $\mathcal{F}$ are invertible and the rank of $\mathcal{F}$ is constant,
\begin{align*}
    \mathcal{F}\big((\mathbf{1},a)\big)\mathcal{F}\big((g,\mathbbm{1}_{\tau_{i(a)}})\big) &= \mathcal{F}\big((\psi_a(g),\mathbbm{1}_{\tau_{t(a)}})\big)\mathcal{F}\big((\mathbf{1},a)\big)\\
    \mathcal{F}\big((g,\mathbbm{1}_{\tau_{i(a)}})\big) &= \mathcal{F}\big((\mathbf{1},a)\big)^{-1}\mathcal{F}\big((\psi_a(g),\mathbbm{1}_{\tau_{t(a)}})\big)\mathcal{F}\big((\mathbf{1},a)\big)
\end{align*}
\end{proof}
\begin{lemma}\label{L: comp 1.b}
Let $\Lambda = (\gs, G_{\sigma}, \psi_a, g_{a,b})$ be a complex of groups, $C\Lambda$ the B.H. category and $\mathcal{F}$ a constant rank sheaf over $\cg$ whose linear morphisms are all invertible. Then for composable arrows $(\mathbf{1},a),(\mathbf{1},b) \in \Arr(C\Lambda)$ we get the relation,
$$
\mathcal{F}\big((\mathbf{1},a)\big)\circ\mathcal{F}\big((\mathbf{1},b)\big) = \mathcal{F}(g_{a,b},\mathbbm{1}_{t(a)})\circ\mathcal{F}\big((\mathbf{1},ab)\big)
$$
\end{lemma}
\begin{proof}
We compose $(\mathbf{1},a)$ and $(\mathbf{1},b)$,
\begin{align*}
    (\mathbf{1},a)(\mathbf{1},b) &= (g_{a,b}, ab)\\
    &= (g_{a,b}, \mathbbm{1}_{t(a)})(\mathbf{1}, ab)
\end{align*}
Functors preserve composition of arrows,
\begin{align*}
    \mathcal{F}\big((\mathbf{1},a)\big)\mathcal{F}\big((\mathbf{1},b)\big) &= \mathcal{F}\big((g_{a,b}, \mathbbm{1}_{t(a)})\big)\mathcal{F}\big((\mathbf{1}, ab)\big)
\end{align*}
\end{proof}
We now give the proof of Lemma \ref{keylemma} the Key Lemma.
\begin{proof}[Proof (Key Lemma)]
 We start by proving the two conditions from Definition \ref{complex of groups morphism}
$$
\phi_{t(a)}\psi_a = \mathrm{Ad}\big(\phi_{\idGrp}(a)\big)\phi_{i(a)} \hspace{.1in}\text{ and }\hspace{.1in} \phi_{t(a)}(g_{a,b})\phi(ab) = \phi(a)\phi(b)
$$

Suppose first $a\in T$. There are two cases,

\begin{align*}
    \textbf{Case 1:} &\\ 
    \phi_{\idGrp}(a) &= \sheafpath{t(a)}^{-1}\mathcal{F}\big((\idGrp,a)\big)\sheafpath{i(a)}\\
    &=\sheafpath{i
    (a)}^{-1}\mathcal{F}\big((\idGrp,a)\big)^{-1}\mathcal{F}\big((\idGrp,a)\big)\sheafpath{i(a)} \\
    &= \idGrp_V
\end{align*}
\begin{align*}
    \textbf{Case 2:} &\\ 
    \phi_{\idGrp}(a) &= \sheafpath{t(a)}^{-1}\mathcal{F}\big((\idGrp,a)\big)\sheafpath{i(a)}\\
    &=\sheafpath{t
    (a)}^{-1}\mathcal{F}\big((\idGrp,a)\big)\mathcal{F}\big((\idGrp,a)\big)^{-1}\sheafpath{t(a)} \\
    &= \idGrp_V
\end{align*}

For $a\not\in T$, let $g\in G_{i(a)}$. The left hand side of the first condition reads
\begin{align*}
    \phi_{t(a)}\big(\psi_a(g)\big) &= M_{t(a)}^{-1}\mathcal{F}\big((\psi_a(g),\mathbbm{1}_{t(a)})\big) M_{t(a)} \\
    &= M_{t(a)}^{-1}\mathcal{F}\big((\idGrp,a)\big)\mathcal{F}\big((g,\mathbbm{1}_{i(a)})\big)\mathcal{F}\big((\mathbf{1},a)\big)^{-1}M_{t(a)}\\
    &= M_{t(a)}^{-1}\mathcal{F}\big((\mathbf{1},a)\big)M_{i(a)} M_{i(a)}^{-1}\mathcal{F}\big((g,\mathbbm{1}_{i(a)})\big)M_{i(a)} M_{i(a)}^{-1}\mathcal{F}\big((\mathbf{1},a)\big)^{-1}M_{t(a)}
    \\
    &= \phi_{\idGrp}(a)\phi_{i(a)}(g)\phi_{\idGrp}(a)^{-1}\\
    &= \mathrm{Ad}\big(\phi_{\idGrp}(a)\big)\phi_{i(a)}(g)
\end{align*}
The second line is justified by Lemma \ref{L: comp 1.a}

For composable edges $a,b\not\in \Arr(T)$, $i(a) = t(b)$, we observe the right hand side of the second equation,
\begin{align*}
    \phi_{\idGrp}(a)\phi_{\idGrp}(b) &= M_{t(a)}^{-1}\mathcal{F}\big((\idGrp,a)\big)M_{i(a)}M_{t(b)}^{-1}\mathcal{F}\big((\idGrp,b)\big)M_{i(b)}\\
    &= M_{t(a)}^{-1}\mathcal{F}\big((\mathbf{1},a)\big)\mathcal{F}\big((\idGrp,b)\big)M_{i(b)}\\
    &= M_{t(a)}^{-1}\mathcal{F}(g_{a,b},\mathbbm{1}_{t(a)})\mathcal{F}\big((\mathbf{1},ab)\big)M_{i(b)}\\
    &= M_{t(a)}^{-1}\mathcal{F}(g_{a,b},\mathbbm{1}_{t(a)})M_{t(a)}M_{t(a)}^{-1}\mathcal{F}\big((\idGrp,ab)\big)M_{i(b)} \\
    &= \phi_{t(a)}(g_{a,b})\phi_{\idGrp}(ab)
\end{align*}
The third line is justified by Lemma \ref{L: comp 1.b}. 

We make the further assumption that $\mathcal{F}$ satisfies the dev-properties, by Lemma \ref{local group hom} each $\phi_{\sigma}$ is injective. Hence, $\phi$ is injective on local groups.
\end{proof}

This gives the other direction of Theorem \ref{main}
\begin{theorem}\label{otherdirection}
Let $\cg$ be a complex of groups and $\mathcal{F}$ a sheaf over $\cg$ that satisfies the dev-properties with witness $\GL(V)$. Then $\cg$ is developable. 
\end{theorem}
\begin{proof}
 By Lemma \ref{keylemma} the map $\phi = (\phi_{\sigma}, \phi_{\mathbf{1}}(a))$ constructed in Definition \ref{keydef} is a morphism of complexes of groups and by the same lemma is injective on local groups if $\mathcal{F}$ satisfies the dev-properties. Thus, by Theorem \ref{key theorem}, $\cg$ is developable. 
\end{proof}
 \subsection{Theorem \ref{main}}\label{proofA}
 \Dev*
 \begin{proof}
  One direction is given by Theorem \ref{otherdirection} while the other is the content of Lemma \ref{onedirection}.
 \end{proof}
\printbibliography
\end{document}